\documentclass[12pt]{amsart} 

\usepackage[utf8]{inputenc}

\title{The Generalized Terwilliger Algebra of the Hypercube}
\date{\today}
\author{Nathan Nicholson}

\usepackage{amsfonts,latexsym,wasysym, mathrsfs, mathtools,hhline,color, bm}
\usepackage[all, cmtip]{xy}

\usepackage{cite}

\usepackage{url}

\definecolor{hot}{RGB}{65,105,225}

\usepackage{ textcomp }
\usepackage{ tipa }

\usepackage{graphicx,enumerate}

\usepackage{enumitem}
\usepackage[normalem]{ulem}

 \usepackage{amsthm}  
 \usepackage{amsmath} 
 \usepackage{amsfonts}  
 \usepackage{amssymb}
 \usepackage{amscd} 
 \usepackage[all]{xy} 
\usepackage[dvipsnames]{xcolor}  
\usepackage{graphicx}
\usepackage{pdfsync}
\usepackage{url}
\usepackage{tikz}
\usepackage{tikz-cd}
\usepackage{comment}

\newtheorem{thm}{Theorem}[section]

\theoremstyle{plain}
\newtheorem{ex}[thm]{Example}

\newtheorem{lem}[thm]{Lemma}
\newtheorem{prop}[thm]{Proposition}
\newtheorem{cor}[thm]{Corollary}
\theoremstyle{definition}
\newtheorem{defn}[thm]{Definition}
\newtheorem{rem}[thm]{Remark}

\newcommand{\C}{\mathbb{C}}

\newcommand{\cT}{\mathcal{T}}
\newcommand{\cQ}{\mathcal Q}
\newcommand{\cC}{\mathcal C}
\newcommand{\fT}{\mathfrak T}
\newcommand{\fS}{\mathfrak S}
\newcommand{\fQ}{\mathfrak K}
\newcommand{\fe}{\mathfrak e}
\newcommand{\fx}{\mathfrak x}

\newcommand{\renumerate}{\begin{enumerate}[label=(\roman*)]}
\newcommand{\eenumerate}{\end{enumerate}}

\begin{document}

\begin{abstract}
In the year 2000, Eric Egge introduced the generalized Terwilliger algebra $\cT$ of a distance-regular graph $\Gamma$. For any vertex $x$ of $\Gamma,$ there is a surjective algebra homomorphism $\natural$ from $\cT$ to the Terwilliger algebra $T(x)$. If $\Gamma$ is complete, then $\natural$ is an isomorphism. If $\Gamma$ is not complete, then $\natural$ may or may not be an isomorphism, and in general the details are unknown. We show that if $\Gamma$ is a hypercube, then the algebra homomorphism $\natural:\cT \to T(x)$ is an isomorphism for all vertices $x$ of $\Gamma$.

\bigskip

\noindent
\textbf{Keywords.} Distance-regular graph; Terwilliger algebra; Hamming cube; hypercube.

\noindent
\textbf{2020 Mathematics Subject Classification.} 
05E30. 
\end{abstract}

\maketitle


\section{Introduction}\label{s:introduction}

In the topic of Algebraic Combinatorics, there is a family of finite undirected graphs said to be distance-regular \cite{bcn}. These graphs are heavily studied; see \cite{bannaiIto1, bannaiIto2, bcn, vdk, TerDrgSurvey}. For a vertex $x$ of a distance-regular graph $\Gamma$, the associated Terwilliger algebra $T(x)$ was introduced in \cite{terwSub1} (there called the subconstituent algebra). This algebra is finite-dimensional, semi-simple, and noncommutative in general. Some notable papers about $T(x)$ are \cite{caughman, curtin1, curtin2, dickie, hobartIto, pascasio, tanabe, tanaka, terwSub1, terwSub2, terwSub3}. There are some well-known relations in $T(x)$ called the triple product relations; the form of the triple product relations is independent of $x$ (see \cite[Lemma 3.2]{terwSub1}).

In \cite{egge1}, Eric Egge introduced the generalized Terwilliger algebra $\cT$ of $\Gamma$. This algebra is defined by generators and relations; the main relations are the triple product relations. By construction, for any vertex $x$ of $\Gamma$ there is a surjective algebra homomorphism $\natural: \cT \to T(x)$. If the graph $\Gamma$ is complete, then $\natural$ is an isomorphism. If $\Gamma$ is not complete, then $\natural$ may or may not be an isomorphism, and in general the details are unknown.

There is a type of distance-regular graph called the hypercube. The hypercube of diameter $d$ is often denoted $\cQ_d$. General information about $\cQ_d$ can be found in \cite[Chapters 1 and 9]{bcn}. In \cite{go}, Junie Go showed that for any vertex $x$ of $\cQ_d$ there is an algebra isomorphism from $T(x)$ to a direct sum of full matrix algebras:
$$T(x) \to \bigoplus_{0 \leq r \leq \lfloor d/2 \rfloor} M_{d+1-2r}(\C).$$
In this paper, we investigate the algebra $\cT$ associated with $\cQ_d$. Our main result is that for any vertex $x$ of $\cQ_d$, the algebra homomorphism $\natural: \cT \to T(x)$ is an isomorphism. To prove this result, we use the following strategy.

Writing $\cT_d$ for the generalized Terwilliger algebra associated with $\cQ_d$, we will display an algebra isomorphism
$$\cT_d \to M_{d+1}(\C) \oplus \cT_{d-2}$$
for $d \geq 2$. We will then argue by induction on $d$.

The paper is organized as follows. In Sections 2 and 3, we discuss the algebras $T(x)$ and $\cT$ for any distance-regular graph. In Sections 4 and 5, we discuss $T(x)$ and $\cT$ for the hypercube $\cQ_d$. In Section 6, we discuss a certain central idempotent $u_0$ of $\cT$ that was introduced in \cite{egge1}. In Section 7, we prove the main result.


\section{Distance-regular Graphs}\label{s:preliminaries}
In this section, we review some definitions and results concerning distance-regular graphs and the Terwilliger algebra. For more information, we refer the reader to \cite{bcn, TerDrgSurvey, terwSub1}.

Suppose $X$ is a nonempty finite set. Let $M_X(\C)$ denote the algebra consisting of the matrices with rows and columns indexed by $X$ and entries in $\C$. Let $\C^X$ denote the vector space over $\C$ consisting of column vectors with coordinates indexed by $X$ and entries in $\C$. The algebra $M_X(\C)$ acts on $\C^X$ by left multiplication. For any positive integer $n,$ let $M_n(\C)$ denote the algebra consisting of square $n \times n$ matrices with entries in $\C$. 

Let $\Gamma = (X,R)$ denote a finite, undirected, connected graph, without loops or multiple edges, with vertex set $X$, edge set $R$, and path-length distance function $\partial$. Let
$$d = \max\{\partial(x,y) \mid x,y \in X\}.$$
We call $d$ the \textit{diameter} of $\Gamma$. Vertices $x,y \in X$ are said to be \textit{adjacent} whenever they form an edge. For any vertex $x \in X,$ let $\Gamma(x)$ denote the set of vertices that are adjacent to $x$.

Define $A \in M_X(\C)$ with $(x,y)$-entry
$$A_{xy} = \begin{cases} 1 & \text{if $x$ and $y$ are adjacent}; \\ 0 & \text{if $x$ and $y$ are not adjacent}\end{cases} \qquad \qquad (x,y \in X).$$
We call $A$ the \textit{adjacency matrix of $\Gamma$}. 

We say that $\Gamma$ is \textit{regular with valency $k$} whenever $k = |\Gamma(x)|$ for all $x \in X.$

We say that $\Gamma$ is \textit{distance-regular} whenever for any $0\leq h,i,j \leq d$, and any vertices $x,y \in X$ with $\partial(x,y) = h,$ the number of vertices $z \in X$ such that $\partial(x,z) = i$ and $\partial(y,z) = j$ is a constant depending only on $h,i,j$, but not on $x$ and $y$. We denote this constant by $p^h_{ij}$ and call it an \textit{intersection number of $\Gamma$}. Observe that $p^h_{ij} = p^h_{ji}$ $(0 \leq h,i,j \leq d).$

For the rest of this section, assume $\Gamma$ is distance-regular. Note that $\Gamma$ is regular with valency $k = p^0_{11}$ if $d \geq 1$.

We now recall the Bose-Mesner algebra of $\Gamma$. For $0 \leq i \leq d$, define $A_i \in M_X(\C)$ with $(x,y)$-entry
$$
(A_i)_{xy}= \begin{cases} 1 & \text{if }\partial(x,y) = i; \\ 0 & \text{if }\partial(x,y) \neq i \end{cases}\qquad \qquad (x,y \in X).
$$
We call $A_i$ the \textit{$i$th distance matrix of $\Gamma$}. Note that $A_0 = I,$ where $I \in M_X(\C)$ denotes the identity matrix. Note that $A_1 = A$, provided that $d \geq 1$. The sum $\sum_{i=0}^d A_i = J,$ where $J \in M_X(\C)$ denotes the matrix with every entry equal to 1. For notational convenience, define
\begin{equation} 
\label{ai=0_if}
A_i = 0 \qquad \qquad (i < 0 \text{ or }i > d).\end{equation}

By \cite[p.~44]{bcn},
\begin{equation}
A_iA_j = \sum_{h=0}^d p^h_{ij}A_h \qquad \qquad (0 \leq i,j \leq d).
\label{AiAjformula}
\end{equation}
Thus $\{A_i\}_{i =0}^d$ forms a basis for a commutative subalgebra $M$ of $M_X(\C)$. The matrix $A$ generates $M$ by \cite[p.~127]{bcn}. We call $M$ the \textit{Bose-Mesner algebra of $\Gamma$}. Note that $M$ has dimension $d+1$.

For $0 \leq i \leq d,$ define $k_i = p^0_{ii}.$ For any $x \in X,$ $k_i$ is equal to the number of vertices at distance $i$ from $x$. We call $k_i$ the \textit{$i$th valency of $\Gamma$}. Note that $k_0 = 1$. Moreover, $k_1 = k$ if $d \geq 1$.

The matrix $A$ has $d+1$ distinct eigenvalues, because $A$ generates $M$ and $M$ has dimension $d+1$. Denote these eigenvalues by $\theta_0 > \theta_1 > \cdots > \theta_d$. For $0 \leq i \leq d$, let $E_i \in M_X(\C)$ be the matrix which acts as $I$ on the $\theta_i$-eigenspace of $A$ and as 0 on all other eigenspaces of $A$. The matrices $\{E_i\}_{i=0}^d$ form a basis for $M$, and
\begin{equation}\label{e:ei_formulas}\sum_{i=0}^d E_i = I, \qquad \qquad E_iE_j = \delta_{ij}E_i \qquad \qquad (0 \leq i,j \leq d).\end{equation}
We call $E_i$ the \textit{primitive idempotent of $\Gamma$ associated with $\theta_i$} ($0 \leq i \leq d$). Define
\begin{equation} E_i = 0 \qquad \qquad (i < 0 \text{ or } i > d). \end{equation}

By \cite[p.~45]{bcn}, we have $E_0 = \vert X \vert^{-1}J$. By construction,
\begin{equation}
    A = \sum_{i=0}^d \theta_iE_i.
    \label{A1_sum_of_Ei}
\end{equation}

Next we recall the Krein parameters of $\Gamma$. For $B,C \in M_X(\C),$ let $B \circ C$ denote their entry-wise product. Note that
$$A_i \circ A_j = \delta_{ij}A_i \qquad \qquad (0 \leq i,j \leq d).$$ Consequently, $M$ is closed under $\circ$. Because $\{E_i\}_{i=0}^d$ is a basis for $M$, there exist scalars $(p^{h}_{ij})^* \in \C$ ($0 \leq h,i,j \leq d)$ such that
\begin{equation}
E_i \circ E_j = |X|^{-1}\sum_{h=0}^d (p^{h}_{ij})^*E_h \qquad \qquad (0 \leq i,j \leq d).
\label{krein_parameter_def}
\end{equation}
The scalars $(p^{h}_{ij})^*$ are called the \textit{Krein parameters of $\Gamma$}. By \cite[p.~50]{bcn}, $(p^{h}_{ij})^*$ is real and nonnegative for $0 \leq h,i,j \leq d$.

For $0 \leq i \leq d,$ define $k_i^* = (p^{0}_{ii})^*$. We call $k_i^*$ the \textit{$i$th dual valency of $\Gamma$}.

We next recall the dual Bose-Mesner algebras of $\Gamma$. Fix a vertex $x \in X.$ For $0 \leq i \leq d,$ let $E_i^*=E_i^*(x)$ denote the diagonal matrix in $M_X(\C)$ with $(y,y)$-entry
$$(E_i^*)_{yy} = \begin{cases}
1 & \text{if }\partial(x,y) = i; \\
0 & \text{if }\partial(x,y) \neq i
\end{cases} \qquad \qquad (y \in X).$$
We call $E_i^*$ the \textit{ith dual primitive idempotent of $\Gamma$ with respect to $x$}. Define
\begin{equation}E_i^* = 0 \qquad \qquad (i < 0 \text{ or } i > d).\end{equation}

By construction,
$$\sum_{i=0}^d E_i^* = I, \qquad \qquad E_i^*E_j^* = \delta_{ij}E_i^* \qquad \qquad (0 \leq i,j \leq d).$$
It follows that $\{E_i^*\}_{i=0}^d$ forms a basis for a commutative subalgebra $M^* = M^*(x)$ of $M_X(\C).$ We call $M^*$ the \textit{dual Bose-Mesner algbebra of $\Gamma$ with respect to $x$}.

For $0 \leq i \leq d,$ we define a diagonal matrix $A_i^* = A_i^*(x) \in M_X(\C)$ with $(y,y)$-entry
$$(A_i^*)_{yy} = |X| (E_i)_{xy} \qquad \qquad (y \in X).$$
Note that $A_0^* = I.$ For $d \geq 1,$ we abbreviate $A^* = A_1^*$ and call this the \textit{dual adjacency matrix of $\Gamma$ with respect to $x$}. Define
\begin{equation}
\label{e:ai=0_if_i_bad}
A_i^* = 0 \qquad \qquad (i < 0 \text{ or } i > d).\end{equation}
By \cite[Corollary~11.6]{TerDrgSurvey}, $A^*$ generates $M^*$.

By \eqref{krein_parameter_def},
$$A_i^*A_j^* = \sum_{h=0}^d (p^{h}_{ij})^*A_h^* \qquad \qquad (0 \leq i,j \leq d).$$
By \cite[Lemma~5.8]{TerDrgSurvey}, the matrices $\{A_i^*\}_{i=0}^d$ form a basis for $M^*.$

We next recall the Terwilliger algebra of $\Gamma$. Let $T = T(x)$ denote the subalgebra of $M_X(\C)$ generated by $M$ and $M^*.$ We call $T$ the \textit{Terwilliger algebra of $\Gamma$ with respect to $x$}. We remark that $T$ is sometimes called the subconstituent algebra.

For $d \geq 1,$ since the matrices $\{E_i^*\}_{i=0}^d$ form a basis for $M^*,$ there exist scalars $\theta_i^* \in \C$ ($0 \leq i \leq d$) such that
$$A^* = \sum_{i=0}^d \theta_i^*E_i^*.$$

Because $\{A_i\}_{i=0}^d$ and $\{E_i\}_{i=0}^d$ each constitute bases for $M$, there must exist scalars $p_i(j),q_i(j) \in \C$ ($0 \leq i,j \leq d$) such that
\begin{align}
    A_i = \sum_{j=0}^d p_i(j)E_j,  \qquad \qquad E_i &= {|X|}^{-1}\sum_{j=0}^d q_i(j)A_j \qquad \qquad (0 \leq i \leq d).
    \label{e:pij_and_qij}
\end{align}
Similarly, there must exist scalars $p_i^*(j),q_i^*(j) \in \C$ ($0 \leq i,j \leq d$) such that
\begin{align*}
    A_i^* = \sum_{j=0}^d p_i^*(j)E_j^*,\qquad \qquad E_i^* = {|X|}^{-1}\sum_{j=0}^d q_i^*(j)A_j^* \qquad \qquad (0 \leq i \leq d).
\end{align*}
By the construction of $\{E_i^*\}_{i=0}^d$ and $\{A_i^*\}_{i=0}^d$, we have
\begin{equation}
    q_i(j) = p_i^*(j), \qquad \qquad p_i(j) = q_i^*(j) \qquad \qquad (0 \leq i,j \leq d).
    \label{qi_equals_pi*}
\end{equation}
For $d \geq 1$, we have
\begin{equation}
    \theta_i = p_1(i), \qquad \qquad \theta_i^* = p_1^*(i) \qquad \qquad (0 \leq i \leq d).
    \label{thetai_equals_p1}
\end{equation}
By \cite[Lemma 2.2.1]{bcn},
\begin{equation}
    p_0(j) = 1, \qquad \qquad q_0(j) = 1 \qquad \qquad (0 \leq j \leq d).
    \label{e: p0(j)=1}
\end{equation}

Define matrices $P,Q \in M_{d+1}(\C)$ with entries $P_{ij} = p_j(i)$ and $Q_{ij} = q_j(i)$ ($0 \leq i,j \leq d$). Then $P$ is the change of basis matrix from $\{A_i\}_{i=0}^d$ to $\{E_i\}_{i=0}^d$, and $|X|^{-1}Q$ is the change of basis matrix from $\{E_i\}_{i=0}^d$ to $\{A_i\}_{i=0}^d$. Moreover, $P$ and $|X|^{-1}Q$ are inverses.

By \eqref{qi_equals_pi*}, $|X|^{-1}P$ is the change of basis matrix from $\{E_i^*\}_{i=0}^d$ to $\{A_i^*\}_{i=0}^d$, and $Q$ is the change of basis matrix from $\{A_i^*\}_{i=0}^d$ to $\{E_i^*\}_{i=0}^d$.

By \cite[Lemma~3.2]{terwSub1}, the following hold for $0 \leq h,i,j \leq d$:
\begin{align*}
    E_h^*A_iE^*_j = 0 & \qquad \text{iff} \qquad p^h_{ij} = 0, \\
    E_hA^*_iE_j = 0 & \qquad \text{iff} \qquad (p^{h}_{ij})^* = 0.
\end{align*}
The above statements are referred to as the \textit{triple product relations}.

To conclude this section, we recall the notion of self-duality. We say that $\Gamma$ is \textit{self-dual} whenever $p^h_{ij} = (p^{h}_{ij})^*$ ($0 \leq h,i,j \leq d$). In this case, we have $k_i = k_i^*$ and $\theta_i = \theta_i^*$ ($0 \leq i \leq d$), and we have $p_{i}(j) = q_i(j)$ ($0 \leq i,j \leq d$). See \cite[p.~49]{bcn}.


\section{The Generalized Terwilliger Algebra}\label{s:Gen_T_Algebra}

In this section, we recall the generalized Terwilliger algebra and some related results from \cite{egge1}. Let $\Gamma = (X,R)$ denote a distance regular graph with diameter $d$. We fix a vertex $x \in X$ and let $T = T(x)$.

\begin{defn}\label{cTDef}
(See \cite[Definition~4.1]{egge1}) Let $\cT$ denote the algebra with 1, with generators $x_0,\dots,x_d,$ $x_0^*,\dots,x_d^*$ and relations $(T1)$--$(T3^*)$.
\begin{align*}
    (T1) & \qquad x_0 = x_0^*=1. \\
    (T2) & \qquad x_ix_j = \sum_{h=0}^d p^h_{ij}x_h & (0 \leq i,j \leq d).\\
    (T2^*) & \qquad x^*_ix^*_j = \sum_{h=0}^d (p^{h}_{ij})^*x_h^*& (0 \leq i,j \leq d).\\
    (T3) & \qquad e_h^*x_ie^*_j = 0 \text{ if } p^h_{ij} = 0 & (0 \leq h,i,j \leq d).\\
    (T3^*) & \qquad e_hx_i^*e_j = 0 \text{ if } (p^{h}_{ij})^* = 0 & (0 \leq h,i,j \leq d).
\end{align*}
In the above lines, we define
\begin{equation}
e_i = |X|^{-1}\sum_{j=0}^d q_i(j)x_j, \qquad \qquad e_i^* = |X|^{-1}\sum_{j=0}^d q^*_i(j)x^*_j \qquad \qquad (0 \leq i \leq d).
\label{eq:ei_def}
\end{equation}
We call $\cT$ the \textit{generalized Terwilliger algebra associated with $\Gamma$}.
\end{defn}

For notational convenience, define
\begin{align}
    x_i = 0, \qquad \qquad x_i^* = 0, \qquad \qquad e_i = 0, \qquad \qquad e_i^* = 0 \qquad \qquad  (i < 0 \text{ or } i > d).
    \label{xi_less_0}
\end{align}

\begin{rem}\label{T0_is_C}
If $d = 0$, the algebra $\cT$ is isomorphic to $\C$.
\end{rem}

\begin{rem}\label{Egge_typo}
In \cite[Definition~4.1]{egge1}, the relation (T1) is given as $x_0=x_0^*$. In \cite[Proposition~5.1]{egge1}, a proof is given that $x_0=x_0^*=1$. However the proof is not correct, and this can be seen by considering the case $d=0.$ Thus we adjusted our statement of $(T1)$ to incorporate the author's assumption that $x_0 = x_0^* = 1.$
\end{rem}

Next we recall some results about $\cT$.

\begin{lem}\label{hom_exists}
{\rm (See \cite[p.~3]{egge1})} There exists an algebra homomorphism $\natural:\cT \to T$ which sends
\begin{align*}
    x_i &\mapsto A_i, \qquad \qquad x_i^* \mapsto A_i^*, \\
    e_i &\mapsto E_i, \qquad \qquad e_i^* \mapsto E_i^*,
\end{align*}
for $0 \leq i \leq d$. Moreover, $\natural$ is surjective.
\end{lem}

The map $\natural$ in Lemma~\ref{hom_exists} is not an isomorphism in general.
However we do have the following results.

\begin{lem}\label{commutative_subalgebras}
{\rm (See \cite[Propositions~5.4 and~10.2]{egge1})} The following (i)--(iii) hold.
\renumerate
    \item The elements $\{x_i\}_{i=0}^d$ form a basis for a commutative subalgebra $\cC$ of $\cT$.
    \item The elements $\{e_i\}_{i=0}^d$ form a basis for $\cC$.
    \item The restriction of $\natural$ to $\cC$ induces an algebra isomorphism $\cC \to M$.
\eenumerate
\end{lem}

\begin{lem}\label{dual_commutative_subalgebras}
{\rm (See \cite[Propositions~5.7 and~10.3]{egge1})} The following (i)--(iii) hold.
\renumerate
    \item The elements $\{x_i^*\}_{i=0}^d$ form a basis for a commutative subalgebra $\cC^*$ of $\cT$.
    \item The elements $\{e_i^*\}_{i=0}^d$ form a basis for $\cC^*$.
    \item The restriction of $\natural$ to $\cC^*$ induces an algebra isomorphism $\cC^* \to M^*$.
\eenumerate
\end{lem}

The next lemmas are consequences of Lemmas~\ref{commutative_subalgebras} and~\ref{dual_commutative_subalgebras}.

\begin{lem}\label{e_sum_1}
{\rm (See \cite[Propositions~5.3 and~5.6]{egge1})} The following (i), (ii) hold in $\cT$.
\renumerate
    \item $\sum_{i=0}^d e_i = 1$.
    \item $\sum_{i=0}^d e_i^* = 1$.
\eenumerate
\end{lem}

\begin{proof}
Follows from \eqref{e:ei_formulas} and Lemmas~\ref{commutative_subalgebras} and~\ref{dual_commutative_subalgebras}.
\end{proof}

\begin{lem}\label{e_orthogonal}
{\rm (See \cite[Propositions~5.3 and~5.6]{egge1})} For $0 \leq i,j \leq d $, the following (i), (ii) hold in $\cT$.
\renumerate
    \item $e_ie_j = \delta_{ij}e_i$.
    \item $e_i^*e_j^* = \delta_{ij}e_j^*$.
\eenumerate
\end{lem}

\begin{proof}
Follows from \eqref{e:ei_formulas} and Lemmas~\ref{commutative_subalgebras} and~\ref{dual_commutative_subalgebras}.
\end{proof}

\begin{lem}\label{x1_sum_eigenvalues}
{\rm (See \cite[Propositions~5.3 and~5.6]{egge1})} For $d \geq 1$, the following (i), (ii) hold in $\cT$.
\renumerate
    \item $x_1 = \sum_{i=0}^d \theta_i e_i$.
    \item $x_1^* = \sum_{i=0}^d \theta_i^*e_i^*.$
\eenumerate
\end{lem}

\begin{proof}
Follows from \eqref{A1_sum_of_Ei} and Lemmas~\ref{commutative_subalgebras} and~\ref{dual_commutative_subalgebras}.
\end{proof}

\begin{lem}\label{sum_of_xi}
The following (i), (ii) hold in $\cT$.
\renumerate
    \item $e_0 = |X|^{-1}\sum_{i=0}^d x_i$.
    \item $e_0^* = |X|^{-1} \sum_{i=0}^d x_i^*$.
\eenumerate
\end{lem}

\begin{proof}
Follows from~\eqref{qi_equals_pi*},~\eqref{e: p0(j)=1}, and~\eqref{eq:ei_def}.
\end{proof}


\section{The Hypercube $\cQ_d$}\label{s:hypercube}

In this section, we recall a family of distance-regular graphs called the hybercubes, and we review some results related to the associated Terwilliger Algebra. For more information, we refer the reader to \cite{bcn,go,sloane}.

\begin{defn}\label{QdDef}
Assume $d \geq 0$. Let $\cQ_d$ denote the graph with vertex set $X$ consisting of $d$-tuples $(a_1,a_2,\ldots,a_d)$ such that $a_i \in \{-1,1\}$ ($1 \leq i \leq d$). Two vertices are adjacent in $\cQ_d$ whenever they differ in exactly one coordinate. The graph $\cQ_d$ is called the \textit{$d$-cube} or a \textit{hypercube}.
\end{defn}

The graph $\cQ_d$ has $2^d$ vertices. Furthermore, $\cQ_d$ is distance-regular (see \cite[Proposition 1.12.1]{bcn}).

From now through Lemma \ref{A1*_minimal_polynomial}, we consider the distance-regular graph $\Gamma = \cQ_d$ with $d \geq 1$. We now recall the intersection numbers of $\cQ_d$.

\begin{prop}\label{intnumformulua}
{\rm (See \cite[p. 238]{sloane})} For $0 \leq h,i,j \leq d$ we have
$$p^h_{ij} = \begin{cases}
0 & \text{if $h+i+j$ is odd;} \\
\binom{h}{(i-j+h)/2}\binom{d-h}{(i+j-h)/2} & \text{if $h+i+j$ is even}.
\end{cases}$$
In the lines above, we interpret $\binom{n}{m} = 0$ if $m < 0$ or $m > n.$
\end{prop}

We mention some consequences of Proposition~\ref{intnumformulua}.

\begin{cor}
For $0 \leq i \leq d$ we have $k_i = \binom{d}{i}.$
\label{k0_and_kd}
\end{cor}

\begin{proof}
Follows from Proposition~\ref{intnumformulua}.
\end{proof}

\begin{cor}\label{ph1j_value}
For $0 \leq h,i \leq d$ we have
$$p^h_{1i} = \begin{cases}
d-h & \text{if }h = i-1; \\
h & \text{if }h = i+1; \\
0 & \text{if }h \neq i \pm 1.
\end{cases}$$
\end{cor}

\begin{proof}
Follows from Proposition~\ref{intnumformulua}.
\end{proof}

\begin{cor}\label{recurrence}
For $0 \leq i \leq d$ we have
\begin{align*}
    AA_{i} &= (i+1)A_{i+1} + (d-i+1)A_{i-1},
\end{align*}
where we recall \eqref{ai=0_if}.
\end{cor}

\begin{proof}
Follows from \eqref{AiAjformula} and Corollary~\ref{ph1j_value}.
\end{proof}

Let $z$ denote an indeterminate, and let $\C[z]$ denote the algebra of polynomials in $z$ that have coefficients in $\C$. Motivated by Corollary~\ref{recurrence}, we now define some polynomials in $\C[z]$.

\begin{defn}\label{Fi_def}
Let $\{F_i\}_{i=0}^{d+1}$ denote polynomials in $\C[z]$ such that $F_0 = 1$, $F_1 = z,$ and
$$zF_i = (i+1)F_{i+1} + (d-i+1)F_{i-1} \qquad \qquad (1 \leq i \leq d).$$
\end{defn}

We remark that the polynomial $F_i$ has degree $i$ and leading coefficient $\frac{1}{i!}$ ($0 \leq i \leq d+1$).

\begin{lem}\label{Ai_is_polynomial}
We have $F_i(A) = A_i$ ($0 \leq i \leq d$), and $F_{d+1}(A) = 0$. Furthermore, $p_i(j) = F_i(\theta_j)$ ($0 \leq i,j \leq d$).
\end{lem}

\begin{proof}
By construction, $F_0(A) = I = A_0$ and $F_1(A) = A = A_1.$ To see that $A_i = F_i(A)$ ($2 \leq i \leq d$) and $F_{d+1}(A) = 0,$ use induction and compare Corollary~\ref{recurrence} with Definition~\ref{Fi_def}.

To see that $F_i(\theta_j) = p_i(j)$, apply $F_i$ to both sides of \eqref{A1_sum_of_Ei} and simplify with \eqref{e:ei_formulas} to obtain
$$F_i(A) = \sum_{j=0}^d F_i(\theta_j)E_j.$$
Comparing the above with \eqref{e:pij_and_qij}, it follows that $F_i(\theta_j) = p_i(j)$.
\end{proof}

We next consider the eigenvalues of $\cQ_d$.

\begin{lem}\label{eigenvalue_formula}
{\rm (See \cite[Proposition 9.2.1]{bcn})} For $0 \leq i \leq d$ we have $\theta_i = d-2i$.
\end{lem}

The following definition is for notational convenience.

\begin{defn}\label{Phi_def}
Define $\Phi_d \in \C[z]$ by
\begin{equation}
    \label{e:definition_of_phid}
    \Phi_d = \prod_{i=0}^d\big(z-(d-2i)\big).
\end{equation}
\end{defn}

\begin{lem}\label{A1_minimal_polynomial}
The minimal polynomial of $A$ is equal to $\Phi_d$.
\end{lem}

\begin{proof}
Immediate from Lemma~\ref{eigenvalue_formula} and Definition~\ref{Phi_def}.
\end{proof}

\begin{lem}\label{Fd+1_equals_phi}
We have $\Phi_d = (d+1)!F_{d+1}$.
\end{lem}

\begin{proof}
By Lemma~\ref{Ai_is_polynomial}, $F_{d+1}(A)= 0$. By Lemma~\ref{A1_minimal_polynomial}, $\Phi_d$ must divide $F_{d+1}$ in $\C[z]$. Because both $\Phi_d$ and $F_{d+1}$ have degree $d+1$, one must be a scalar multiple of the other. Comparing the leading coefficients, the result follows.
\end{proof}

We have a comment.

\begin{lem}\label{Qd_is_selfdual}
{\rm (See \cite[p. 194]{bcn})} The hypercube $\cQ_d$ is self-dual. In other words, $p^h_{ij} = (p^{h}_{ij})^*$ ($0 \leq h,i,j \leq d$).
\end{lem}

\begin{cor}
For $0 \leq i \leq d$ we have $k_i^* = \binom{d}{i}$.
\label{k0*_and_kd*}
\end{cor}

\begin{proof}
Follows from Corollary~\ref{k0_and_kd} and Lemma~\ref{Qd_is_selfdual}.
\end{proof}

Next we state some corollaries of Lemma~\ref{Qd_is_selfdual}. For the rest of this section, fix a vertex $x$ of $\cQ_d$, and let $T = T(x)$.

\begin{cor}\label{dual_recurrence}
For $0 \leq i \leq d$ we have
\begin{align*}
    A^*A_{i}^* &= (i+1)A_{i+1}^* + (d-i+1)A_{i-1}^*,
\end{align*}
where we recall \eqref{e:ai=0_if_i_bad}.
\end{cor}

\begin{proof}
Similar to the proof of Corollary~\ref{recurrence}.
\end{proof}

\begin{cor}\label{Ai*_is_polynomial}
We have $F_i(A^*) = A_i^*$ ($0 \leq i \leq d$), and $F_{d+1}(A^*) = 0$. Furthermore, $p_i^*(j) = F_i(\theta_j^*)$ ($0 \leq i,j \leq d$).
\end{cor}

\begin{proof}
Similar to the proof of Lemma~\ref{Ai_is_polynomial}.
\end{proof}

\begin{cor}\label{dual_eigenvalue_formula}
For $0 \leq i \leq d$ we have $\theta_i^* = d-2i.$
\end{cor}

\begin{proof}
Follows from Lemma ~\ref{Qd_is_selfdual} and the discussion of self-duality at the end of Section~\ref{s:preliminaries}.
\end{proof}

\begin{lem}\label{A1*_minimal_polynomial}
The minimal polynomial of $A^*$ is equal to $\Phi_d$.
\end{lem}

\begin{proof}
Follows from Corollary~\ref{dual_eigenvalue_formula}.
\end{proof}

From now until the end of the section, we assume $\Gamma = \cQ_d$ with $d \geq 0$.

In the next result, we consider the triples $h,i,j$ such that $p^h_{ij}$ and $(p^{h}_{ij})^*$ are nonzero.

\begin{cor}\label{permissible_conditions}
For $0 \leq h,i,j \leq d,$ the intersection number $p^h_{ij}$ is nonzero if and only if the Krein parameter $(p^{h}_{ij})^*$ is nonzero if and only if the following (i)--(iii) hold.
\renumerate
    \item $h,i,j$ satisfy the triangle inequality:
    $$h \leq i + j, \qquad \qquad i \leq h + j, \qquad \qquad j \leq h + i.$$
    \item $h+i+j \leq 2d$.
    \item $h+i + j$ is even.
\eenumerate
\end{cor}

\begin{proof}
The case of $d = 0$ is trivial. The case of $d \geq 1$ follows upon inspection of Proposition~\ref{intnumformulua} and Corollary~\ref{Qd_is_selfdual}.
\end{proof}

\begin{defn}\label{Pr_def}
Let $P_d$ denote the set consisting of the $3$-tuples of integers $(h,i,j)$ such that $0 \leq h,i,j \leq d$ which satisfy (i)--(iii) of Corollary~\ref{permissible_conditions}.
\end{defn}

\begin{lem}\label{phij_0_if_impermissible}
For $0 \leq h,i,j \leq d$, the following (i)--(iii) are equivalent.
\renumerate
    \item $p^h_{ij} \neq 0$.
    \item $(p^{h}_{ij})^* \neq 0$.
    \item $(h,i,j) \in P_d$.
\eenumerate
\end{lem}

\begin{proof}
Compare Corollary~\ref{permissible_conditions} with Definition~\ref{Pr_def}.
\end{proof}

We conclude this section with a brief definition and a comment about the Terwilliger algebra $T$.

\begin{defn}
Assume $x$ is a vertex of $\cQ_d$. Let $T_d = T_d(x)$ denote the Terwilliger algebra of $\cQ_d$ with respect to $x$.
\end{defn}

\begin{prop}\label{T_sum_of_matrices}
{\rm (See \cite[Theorem 14.14]{go})} There exists an algebra isomorphism
$$T_d \to \bigoplus_{0 \leq r \leq \lfloor d/2\rfloor}M_{d+1-2r}(\C).$$
\end{prop}


\vspace{25pt}


\section{The Generalized Terwilliger Algebra for $\cQ_d$}\label{s:simplify_relations}

In Definition~\ref{cTDef}, we described the generalized Terwilliger
algebra for a distance-regular graph. In this section,
we consider this algebra for the graph $\cQ_d$.

\begin{defn}\label{TdDef}
For $d \geq 0,$ let $\cT_d$ denote the generalized Terwilliger algebra associated with $\cQ_d$.
\end{defn}

By Remark~\ref{T0_is_C}, the algebra $\cT_0$ is ismorphic to $\C$. For the rest of this section, we restrict our attention to the algebra $\cT_d$ with $d \geq 1$.

We next observe some analogues of results from Section~\ref{s:hypercube}.

\begin{lem}\label{recurrence_gen}
For $0 \leq i \leq d$ and with reference to \eqref{xi_less_0}, the following (i), (ii) hold in $\cT_d$.
\renumerate
    \item $x_1x_i = (i+1)x_{i+1} + (d-i+1)x_{i-1}$.
    \item $x_1^*x_i^* = (i+1)x^*_{i+1} + (d-i+1)x^*_{i-1}.$
\eenumerate
\end{lem}

\begin{proof}
Follows from Lemmas~\ref{commutative_subalgebras} and~\ref{dual_commutative_subalgebras} and Corollaries~\ref{recurrence} and~\ref{dual_recurrence}.
\end{proof}

\begin{lem}\label{xi_is_polynomial}
For $0 \leq i \leq d,$ the following (i), (ii) hold in $\cT_d$.
    \renumerate
        \item $x_i = F_i(x_1)$.
        \item $x_i^* = F_i(x_1^*).$
    \eenumerate
Moreover, $\cT_d$ is generated by $x_1$ and $x_1^*$.
\end{lem}

\begin{proof}
Follows from Lemmas~\ref{commutative_subalgebras} and~\ref{dual_commutative_subalgebras} and  Corollaries~\ref{Ai_is_polynomial} and~\ref{Ai*_is_polynomial}.
\end{proof}

\begin{lem}\label{xi_minimal_polynomial}
The polynomial $\Phi_d$ is equal to the minimal polynomial of both $x_1$ and $x_1^*$.
\end{lem}

\begin{proof}
Follows from Lemmas~\ref{commutative_subalgebras}, \ref{dual_commutative_subalgebras}, \ref{A1_minimal_polynomial}, and~\ref{A1*_minimal_polynomial}.
\end{proof}

Now that we have Lemmas~\ref{xi_is_polynomial} and~\ref{xi_minimal_polynomial}, some of the relations in Definition~\ref{cTDef} become redundant, giving us the following, simpler presentation $\cT_d$.

\begin{prop}\label{cTd_simplified}
The algebra $\cT_d$ is isomorphic to the algebra with $1$, with generators $x_1,x_1^*$ and relations (1)--(4).
\begin{enumerate}
    \item[\rm (1)] $\Phi_d(x_1) = 0$.
    \item[\rm (2)] $\Phi_d(x_1^*) = 0$.
    \item[\rm (3)] $e_h^*x_ie_j^* = 0 \text{ if } (h,i,j) \notin P_d  \qquad (0 \leq h,i,j \leq d$).
    \item[\rm (4)] $e_hx_i^*e_j = 0 \text{ if } (h,i,j) \notin P_d \qquad (0 \leq h,i,j \leq d)$.
\end{enumerate}
In the above lines, we define
\begin{align*}
    x_0 &= 1, & x_0^* &= 1, \\
    x_i &= F_i(x_1), &x_i^* &= F_i(x_1^*) & (2 \leq i \leq d), \\
    e_i &= 2^{-d}\sum_{j=0}^d q_i(j)x_j, &e_i^* &= 2^{-d}\sum_{j=0}^d q_i^*(j)x_j^* & (0 \leq i \leq d).
\end{align*}
\end{prop}

\begin{proof}
Compare Definition~\ref{cTDef} with Lemmas~\ref{xi_is_polynomial} and~\ref{xi_minimal_polynomial}.
\end{proof}

We would like to provide another presentation for $\cT_d.$ To do this, we first define the following algebra.

\begin{defn}\label{second_presentation_def}
Let $\cT_d'$ denote the algebra with $1$, with generators $e_0,\dots,e_d,e_0^*,\dots,e_d^*$ and relations (1)--(4).
\begin{enumerate}
    \item[\rm (1)] $\sum_{i=0}^d e_i = 1$.
    \item[\rm (2)] $\sum_{i=0}^d e_i^* = 1$.
    \item[\rm (3)] $e_h^*x_ie_j^* = 0 \text{ if } (h,i,j) \notin P_d \qquad (0 \leq h,i,j \leq d)$.
    \item[\rm (4)] $e_hx_i^*e_j = 0 \text{ if } (h,i,j) \notin P_d \qquad (0 \leq h,i,j \leq d)$.
\end{enumerate}
In the above lines, we define
\begin{align}
    x_0 &= 1, & x_0^* &= 1,\\
    x_1 &= \sum_{i=0}^d (d-2i) e_i, & x_1^* &= \sum_{i=0}^d (d-2i)e_i^*, \label{e:x1_is_sum}\\
    x_i &= F_i(x_1), & x_i^* &= F_i(x_1^*) & (2 \leq i \leq d). \label{e:xi=Fi(x_1)}
\end{align}
\end{defn}

We will soon show that the algebra $\cT_d$ is isomorphic to $\cT_d'.$ We first give some lemmas about $\cT_d'$.

\begin{lem}\label{e_idempotent_in_prime}
For $0 \leq i \leq d$, the following (i), (ii) hold in $\cT_d'$.
\renumerate
    \item $e_ie_j = \delta_{ij}e_i.$
    \item $e_i^*e_j^* = \delta_{ij}e_i^*.$
\eenumerate
\end{lem}

\begin{proof}
(i) First assume that $i \neq j$. The triple $(i,0,j)$ violates Corollary~\ref{permissible_conditions} (i), thus $e_ix_0^*e_j = 0$ by relation (4) of Definition~\ref{second_presentation_def}. As $x_0^* = 1,$ we have $e_ie_j = 0.$

Next assume that $i = j.$ We will show that $e_i^2 = e_i.$ By relation (1) of Definition~\ref{second_presentation_def} and the previous paragraph,
$$e_i = e_i\sum_{\ell=0}^d e_\ell = e_i^2.$$

(ii) Similar to the proof of (i).
\end{proof}

Note that by \eqref{e:x1_is_sum} and Lemma~\ref{e_idempotent_in_prime}, the following equations hold in $\cT_d':$
\begin{equation}
    \big(x_1-(d-2i)\big)e_i = 0, \qquad \qquad \big(x_1^* - (d-2i)\big)e_i^* = 0 \qquad \qquad (0 \leq i \leq d).
    \label{e:x1_minus_theta}
\end{equation}

\begin{lem}\label{x1_minimal_poly_in_prime}
The following (i), (ii) hold in $\cT_d'$.
\renumerate
    \item $\Phi_d(x_1) = 0$.
    \item $\Phi_d(x_1^*) = 0.$
\eenumerate
\end{lem}

\begin{proof}
(i) First note that the term $\big(z-(d-2i)\big)$ is a factor of the right-hand side of \eqref{e:definition_of_phid} ($0 \leq i \leq d$). Thus by Definition~\ref{Phi_def} and \eqref{e:x1_minus_theta},
\begin{equation*}
    \Phi_d(x_1)e_i = 0 \qquad \qquad (0 \leq i \leq d).
    \label{e:phid(x1)ei=0}
\end{equation*}
Hence by relation (1) of Definition~\ref{second_presentation_def},
$$\Phi_d(x_1) = \Phi_d(x_1)\sum_{i=0}^d e_i = 0.$$

(ii) Similar to the proof of (i).
\end{proof}

\begin{lem}\label{fi_equals_ei}
For $0 \leq i \leq d,$ the following (i), (ii) hold in $\cT_d'$.
\renumerate
    \item $e_i = 2^{-d}\sum_{j=0}^d q_i(j)x_j$.
    \item $e_i^* = 2^{-d}\sum_{j=0}^d q_i^*(j)x_j^*$.
\eenumerate
\end{lem}

\begin{proof}
(i) By \eqref{e:xi=Fi(x_1)} and Lemma~\ref{e_idempotent_in_prime},
$$x_i = F_i(x_1) = \sum_{j=0}^d F_i(\theta_j)e_j.$$
Hence by Lemma~\ref{Ai_is_polynomial},
\begin{equation}
    x_i = \sum_{j=0}^d p_i(j)e_j.
    \label{e:xi_is_sum_pi(j)}
\end{equation}
The result follows from \eqref{e:xi_is_sum_pi(j)} and the fact that the matrices $P$, $2^{-d}Q$ from below \eqref{thetai_equals_p1} are inverses.

(ii) Similar to the proof of (i).
\end{proof}

We now show that the algebra $\cT_d$ is isomorphic to $\cT_d'$.

\begin{prop}\label{second_presentation}
There exists a unique algebra isomorphism $\cT_d \to \cT_d'$ that sends
$$x_1 \mapsto x_1, \qquad \qquad x_1^* \mapsto x_1^*.$$
Moreover, this map sends
\begin{align*}
    x_i &\mapsto x_i, &x_i^* &\mapsto x_i^*, \\
    e_i &\mapsto e_i, &e_i^* &\mapsto e_i^*,
\end{align*}
for $0 \leq i \leq d$.
\end{prop}

\begin{proof}
We will first show that there exists an algebra homomorphism $\sigma:\cT_d \to \cT_d'$ which sends $x_1 \mapsto x_1$ and $x_1^* \mapsto x_1^*.$ We will then show that there exists an algebra homomorphism $\tau:\cT_d' \to \cT_d$ which sends $e_i \mapsto e_i$ and $e_i^* \mapsto e_i^*$ ($0 \leq i \leq d$). We will next show that $\sigma$ and $\tau$ are inverses, and hence algebra isomorphisms. We will last show that $\sigma(x_i) = x_i$, $\sigma(x_i^*) = x_i^*,$ $\sigma(e_i) = e_i$, and $\sigma(e_i^*) = e_i^*$ ($0 \leq i \leq d$).

We begin by showing that $\sigma$ exists. For $0 \leq i \leq d,$ let $f_i = 2^{-d}\sum_{j=0}^d q_i(j)x_j \in \cT_d'$ and $f_i^* = 2^{-d}\sum_{j=0}^d q_i^*(j)x_j^* \in \cT_d'$. To show that $\sigma$ exists, it is sufficient to show that in $\cT_d',$
\begin{align}
    \Phi_d(x_1) &= 0, & \Phi_d(x_1^*) &= 0, \label{e:phid(x1)=0}\\
    f_h^*x_if_j^* &= 0, & f_hx_i^*f_j &= 0, \label{e:fhxifj}
\end{align}
for $0 \leq h,i,j \leq d$ such that $(h,i,j) \notin P_d.$ 

Lemma~\ref{x1_minimal_poly_in_prime} implies \eqref{e:phid(x1)=0}. Lemma~\ref{fi_equals_ei} implies that $f_i = e_i$ and $f_i^* = e_i^*$ ($0 \leq i \leq d$). Hence \eqref{e:fhxifj} follows by relations (3) and (4) of Definition~\ref{second_presentation_def}.

We next show $\tau$ exists. Let $y_1 = \sum_{i=0}^d \theta_i e_i \in \cT_d$ and $y_1^* = \sum_{i=0}^d \theta_i^*e_i^* \in \cT_d$. To show that $\tau$ exists, it is sufficient to show that in $\cT_d,$
\begin{align}
    \sum_{\ell=0}^d e_\ell &= 1, & \sum_{\ell=0}^d e_\ell^* &= 1, \label{e:sum_eell=0}\\
    e_h^*F_i(y_1)e_j^* &= 0, & e_hF_i(y_1^*)e_j &= 0,\label{ehyiej}
\end{align}
for $0 \leq h,i,j \leq d$ such that $(h,i,j) \notin P_d$.

Lemma~\ref{e_sum_1} implies \eqref{e:sum_eell=0}. Lemma~\ref{x1_sum_eigenvalues} implies that $x_1 = y_1$ and $x_1^* = y_1^*$. Hence $F_i(y_1) = x_i$ and $F_i(y_1^*) = x_i^*$ ($0 \leq i \leq d$) by Lemma~\ref{xi_is_polynomial}. Thus \eqref{ehyiej} follows by relations (3) and (4) of Proposition~\ref{cTd_simplified}.

We now show that $\sigma$ and $\tau$ are inverses. By Lemma~\ref{x1_sum_eigenvalues}, $\tau(x_1) = x_1$ and $\tau(x_1^*) = x_1^*.$ Thus $\tau \circ \sigma:\cT_d \to \cT_d$ is the identity map. By Lemma~\ref{fi_equals_ei}, $\sigma(e_i) = e_i$ and $\sigma(e_i^*) = e_i^*$ ($0 \leq i \leq d$). Thus $\sigma \circ \tau:\cT_d' \to \cT_d'$ is the identity map. Therefore $\sigma$ and $\tau$ are inverses, and hence algebra isomorphisms.

By construction, $\sigma(x_1) = x_1$ and $\sigma(x_1^*) = x_1^*.$ Thus $\sigma(x_i) = x_i$ and $\sigma(x_i^*) = x_i^*$ ($0 \leq i \leq d$). As noted previously, $\sigma(e_i) = e_i$ and $\sigma(e_i^*) = e_i^*$ ($0 \leq i \leq d$). This completes the proof.
\end{proof}

For the rest of the paper, we identify the algebras $\cT_d$ and $\cT_d'$ via the isomorphism in Proposition~\ref{second_presentation}.

We next define a free algebra.

\begin{defn}\label{fTdDef}
Let $\fT_d$ denote the free algebra with generators $\mathfrak e_0,\dots,\mathfrak e_d,\mathfrak e_0^*,\dots,\fe_d^*$. Define $\fx_0 = 1,$ $\fx_1 = \sum_{i=0}^d \theta_i\fe_i$, and $\fx_i = F_i(\fx_1)$ ($2 \leq i \leq d$). Similarly, define $\fx_0^* =1,$ $\fx_1^* = \sum_{i=0}^d \theta_i^*\fe_i^*$, and $\fx_i^* = F_i(x_1^*)$ ($2 \leq i \leq d$).
\end{defn}

For notational convenience, define
\begin{align}
    \fx_i = 0, \qquad \qquad \fx_i^* = 0, \qquad \qquad \fe_i = 0, \qquad \qquad \fe_i^* = 0 \qquad \qquad (i < 0 \text{ or }i > d).
    \label{fxi_less_0}
\end{align}

\begin{defn}\label{fSd_def}
Let $\fS_d$ denote the two-sided ideal of $\fT_d$ generated by the following (1)--(4).
\begin{enumerate}
    \item $\sum_{i=0}^d\fe_i - 1$,
    \item $\sum_{i=0}^d \fe_i^* - 1$,
    \item $\fe_h^*\fx_i\fe_j^* \text{ such that $(h,i,j) \notin P_d$} \qquad (0 \leq h,i,j \leq d),$
    \item $\fe_h\fx_i^*\fe_j \text{ such that $(h,i,j) \notin P_d$} \qquad (0 \leq h,i,j \leq d).$
\end{enumerate}
\end{defn}

\begin{rem}\label{fTd_hom_to_cTd}
Because $\fT_d$ is free, there exists an algebra homomorphism $\psi_d:\fT_d \to \cT_d$ that sends
\begin{align*}
    \fe_i \mapsto e_i ,\qquad \qquad \fe_i^* \mapsto e_i^* \qquad \qquad (0 \leq i \leq d).
\end{align*}
Comparing Definitions~\ref{second_presentation_def} and~\ref{fSd_def}, it follows that the map $\psi_d$ is surjective with kernel $\fS_d$, and that
$$\fx_i \mapsto x_i, \qquad \qquad \fx_i^* \mapsto x_i^* \qquad \qquad (0 \leq i \leq d).$$
\label{r:psid_exists}
\end{rem}

To end this section, we define an algebra homomorphism that will be useful later in the paper.

\begin{defn}\label{induced_iso_p}
Consider the quotient algebra $\fT_d/\fS_d$. With reference to Remark~\ref{r:psid_exists}, the algebra homomorphism $\psi_d$ induces an algebra isomorphism $\fT_d/\fS_d \to \cT_d$. We denote the inverse of this map by $p_d$.
\end{defn}

\vspace{25pt}


\section{The Primary Central Idempotent of $\cT_d$}\label{baseCase}

We continue our discussion of the algebra $\cT_d$ from Definition~\ref{TdDef}. In \cite{egge1}, Egge defines a certain element $u_0 \in \cT_d$ called the primary central idempotent. Later in the paper, we will use $u_0$ to compute the dimension of $\cT_d$. In this section, we recall the definition of $u_0$ and develop some basic facts about it.

\begin{lem}\label{u0_equality}
{\rm (See \cite[Propositions 11.1 and 11.4]{egge1})} For $d \geq 0,$ following holds in $\cT_d$:
\begin{equation}
2^{d}\sum_{i=0}^d k_i^{-1}e_i^*e_0e_i^* = 2^{d}\sum_{i=0}^d (k_i^*)^{-1}e_ie_0^*e_i.
\label{e:u0_equality}
\end{equation}
This element is central and idempotent.
\end{lem}

\begin{defn}
\label{def:u0}
{\rm (See \cite[Proposition 11.1]{egge1})} Referring to Lemma~\ref{u0_equality}, we define $u_0$ to be the common value expressed in \eqref{e:u0_equality}. We call $u_0$ the \textit{primary central idempotent of $\cT_d$}.
\end{defn}

\begin{prop}\label{idempotentresults2} {\rm (See \cite[Proposition 11.5 and Theorem 12.5]{egge1})} For $d \geq 0,$ the following $(i)$--$(iii)$ hold.
\renumerate
    \item The sum $\cT_d = \cT_d u_0 + \cT_d(1-u_0)$ is direct.
    \item $\cT_d u_0$ and $\cT_d(1-u_0)$ are both two-sided ideals of $\cT_d$.
    \item The algebra $\cT_d u_0$ is isomorphic to $M_{d+1}(\C).$
\eenumerate
\end{prop}

\begin{cor}\label{Td_iso_to_M_and_ideal}
For $d \geq 0$, the algebra $\cT_d$ is isomorphic to the direct sum \\ $M_{d+1}(\C) \oplus \cT_d(1-u_0).$
\end{cor}

\begin{proof}
Follows from Proposition~\ref{idempotentresults2}.
\end{proof}

\begin{cor}\label{Tdu0_quotient}
There exists an algebra isomorphism $\cT_d/\cT_du_0 \to \cT_d(1-u_0)$ that sends
$$e_i + \cT_du_0 \mapsto e_i(1-u_0), \qquad \qquad e_i^* +\cT_du_0 \mapsto e_i^*(1-u_0) \qquad \qquad (0 \leq i \leq d).$$
\end{cor}

\begin{proof}
Follows from Proposition~\ref{idempotentresults2} (i).
\end{proof}

We have some comments about $u_0$.

\begin{lem}\label{u0e0}
For $d \geq 0,$ the following (i)--(iv) hold in $\cT_d$.
\renumerate
    \item $u_0e_0 = e_0$.
    \item $u_0e_d = e_d$.
    \item $u_0e_0^* = e_0^*$.
    \item $u_0e_d^* = e_d^*$.
\eenumerate
\end{lem}

\begin{proof}
(i) By Lemma~\ref{e_orthogonal} (i), Lemma ~\ref{sum_of_xi} (ii), Corollary~\ref{k0*_and_kd*}, and Definition~\ref{def:u0},
$$u_0e_0 = 2^d \bigg(\sum_{r=0}^d (k_r^*)^{-1}e_re_0^*e_r\bigg)e_0 = 2^de_0e_0^*e_0 = e_0\bigg(\sum_{i=0}^d x_i^*\bigg)e_0.$$
For $1 \leq i \leq d$, the triple $(0,i,0)$ does not satisfy Corollary~\ref{permissible_conditions} (i), thus $e_0x_i^*e_0 = 0$ by relation (4) of Definition~\ref{second_presentation_def}. Hence by relation (T1) of Definition~\ref{cTDef},
$$e_0\bigg(\sum_{i=0}^d x_i^*\bigg)e_0 = e_0x_0^*e_0 = e_0^2 = e_0.$$
Therefore $u_0e_0 = e_0.$

(ii) By Lemma~\ref{e_orthogonal} (i), Lemma~\ref{sum_of_xi} (ii), Corollary~\ref{k0*_and_kd*}, and Definition~\ref{def:u0},
$$u_0e_d = 2^d \bigg(\sum_{r=0}^d (k_r^*)^{-1}e_re_0^*e_r\bigg)e_d = 2^de_de_0^*e_d = e_d\bigg(\sum_{i=0}^d x_i^* \bigg)e_d.$$
For $1 \leq i \leq d$, the triple $(d,i,d)$ does not satisfy Corollary~\ref{permissible_conditions} (ii), thus $e_dx_i^*e_d = 0$ by relation (4) of Definition~\ref{second_presentation_def}. Hence by relation (T1) of Definition~\ref{cTDef},
$$e_d\bigg(\sum_{i=0}^d x_i^* \bigg)e_d = e_dx_0^*e_d = e_d^2 = e_d.$$
Therefore $u_0e_d = e_d.$

(iii) Similar to the proof of (i).

(iv) Similar to the proof of (ii).
\end{proof}

We finish this section with a comment about the case $d = 1$.

\begin{prop}\label{dEquals1}
For $d = 1$, the element $u_0 = 1.$ Moreover, the algebra $\cT_1$ is isomorphic to $M_2(\C)$.
\end{prop}

\begin{proof}
Because $d=1$, Lemmas~\ref{e_sum_1} and~\ref{u0e0} imply that
$$u_0 = u_0(e_0+e_1) = e_0+e_1 = 1.$$
Therefore the algebra $\cT_1$ is isomorphic to $M_2(\C)$ by Proposition~\ref{idempotentresults2} part (iii).
\end{proof}


\section{The Main Result}
Recall the algebra homomorphism $\natural:\cT_d \to T_d$ from Lemma 3.4. In this section, we prove that $\natural$ is an algebra isomorphism.

Recall the free algebra $\fT_d$ from Definition~\ref{fTdDef}.

\begin{defn}\label{phiDef}
Assume $d \geq 2.$ Let $\varphi_d:\fT_d \to \fT_{d-2}$ be the algebra homomorphism that sends
\begin{align*}
    \fe_0 &\mapsto 0, &\fe_0^* &\mapsto 0, \\
    \fe_i &\mapsto \fe_{i-1},&\fe_i^* &\mapsto \fe_{i-1}^* \qquad \qquad (1 \leq i \leq d-1), \\
    \fe_d &\mapsto 0, &\fe_d^* &\mapsto 0.
\end{align*}
\end{defn}

Referring to Definition~\ref{phiDef} and using \eqref{fxi_less_0}, we see that $\varphi_d$ sends
\begin{equation} \fe_i \mapsto \fe_{i-1}, \qquad \qquad \fe_i^* \mapsto \fe_{i-1}^* \qquad \qquad (0 \leq i \leq d).\end{equation}

\begin{defn}\label{fQd_def}
Assume $d \geq 2.$ Let $\fQ_d$ denote the two-sided ideal of $\fT_d$ generated by $\fe_0,\fe_d,\fe_0^*,\fe_d^*.$
\end{defn}

\begin{lem}\label{phid_exists}
Assume $d \geq 2.$ The map $\varphi_d$ is surjective with kernel $\fQ_d.$
\end{lem}

\begin{proof}
Routine consequence of Definitions~\ref{phiDef} and~\ref{fQd_def}.
\end{proof}

\begin{lem}\label{phi_of_xi}
Assume $d \geq 2$. For $0 \leq i \leq d$, the following (i), (ii) hold.
\renumerate
    \item $\varphi_d(\fx_i) = \begin{cases}
    \fx_i & \text{if }i = 0 \text{ or } i = 1; \\
    \fx_i- \fx_{i-2} & \text{if } 2 \leq i \leq d-2; \\
    \frac{\Phi_{d-2}(\fx_1)}{(d-1)!} - \fx_{i-2} & \text{if } i = d-1; \\
    \frac{\fx_1\Phi_{d-2}(\fx_1)}{d!} - \fx_{i-2} & \text{if }i = d.
    \end{cases}$
    \item $\varphi_d(\fx_i^*) = \begin{cases}
    \fx_i^* & \text{if }i = 0 \text{ or }i = 1; \\
    \fx_i^*- \fx_{i-2}^* & \text{if } 2 \leq i \leq d-2; \\
    \frac{\Phi_{d-2}(\fx_1^*)}{(d-1)!} - \fx_{i-2}^* & \text{if } i = d-1; \\
    \frac{\fx_1^*\Phi_{d-2}(\fx_1^*)}{d!} - \fx_{i-2}^* & \text{if }i = d.
    \end{cases}$
\eenumerate
\end{lem}

\begin{proof}
(i) We begin with a comment. Note that by Definitions~\ref{Fi_def} and~\ref{fTdDef}, the following holds in $\fT_d$:
\begin{equation}
    j\fx_j = {\fx_1\fx_{j-1}-(d-j+2)\fx_{j-2}} \qquad \qquad (2 \leq j \leq d).
    \label{fxi_in_fTd}
\end{equation}
We now consider the cases for $i$.

First, assume $i = 0$. The result holds, because $\fx_0 = 1$. Next, assume $i = 1$. Then by Lemma~\ref{eigenvalue_formula}, Definition~\ref{fTdDef}, and Definition~\ref{phiDef},
\begin{align*}
    \varphi_d(\fx_1) &= \varphi_d\Big(\sum_{j=0}^d (d-2j)\fe_j\Big) \\
    &= \sum_{j=1}^{d-1}(d-2j)\fe_{j-1} \\
    &= \sum_{j=0}^{d-2}(d-2-2j)\fe_j \\
    &= \fx_1.
\end{align*}

By \eqref{fxi_less_0}, it is correct to say that $\varphi_d$ sends $\fx_i \mapsto \fx_i -\fx_{i-2}$ for $i = 0$ and $i = 1$. This allows us to use induction for $2 \leq i \leq d-2$.  We proceed by induction on $i$.

Assume $2 \leq i \leq d-2$. Setting $j = i$ in \eqref{fxi_in_fTd}, applying $\varphi_d$ to both sides, using induction, and dividing by $i$, we obtain
\begin{equation}
\varphi_d(\fx_i) = \frac{\fx_1(\fx_{i-1} - \fx_{i-3}) - (d-i+2)(\fx_{i-2} - \fx_{i-4})}{i}.
\label{phid_of_fxi}
\end{equation}
Using Definitions~\ref{Fi_def} and~\ref{fTdDef}, we find that in $\fT_{d-2}$,
\begin{equation}
    \fx_1\fx_{i-1} = i\fx_i + (d-i)\fx_{i-2}, \qquad \qquad \fx_1\fx_{i-3} = (i-2)\fx_{i-2} +(d-i+2)\fx_{i-4}.
    \label{fxi_and_fxi-2}
\end{equation}
In equation \eqref{phid_of_fxi}, we distribute terms in the numerator, then eliminate $\fx_1\fx_{i-1}$ and $\fx_1\fx_{i-3}$ via \eqref{fxi_and_fxi-2}. This yields $\varphi_d(\fx_i) = \fx_i - \fx_{i-2}.$

Next, assume $i = d-1$. Setting $j = d-1$ in \eqref{fxi_in_fTd}, applying $\varphi_d$ to the result, using induction, and dividing by $d-1$ yields
\begin{equation}
\varphi_d(\fx_{d-1}) = \frac{\fx_1(\fx_{d-2} - \fx_{d-4}) - 3(\fx_{d-3} - \fx_{d-5})}{d-1}.
\label{phid_of_fxd-1}
\end{equation}
Using Definitions~\ref{Fi_def} and~\ref{fTdDef}, we find that in $\fT_{d-2},$
\begin{equation}
    \fx_1\fx_{d-4} = (d-3)\fx_{d-3} + 3\fx_{d-5}.
    \label{fxd-1}
\end{equation}
In \eqref{phid_of_fxd-1}, we distribute terms in the numerator and eliminate $\fx_1\fx_{d-4}$ via \eqref{fxd-1}. This yields
\begin{equation}
    \varphi_d(\fx_{d-1}) = \frac{\fx_1\fx_{d-2} - \fx_{d-3}}{d-1}-\fx_{d-3}.
    \label{e:phid_of_xd-1}
\end{equation}
By Lemma~\ref{Fd+1_equals_phi},
\begin{equation}
    \fx_1\fx_{d-2} - \fx_{d-3} = \frac{\Phi_{d-2}(\fx_1)}{(d-2)!}.
    \label{e:x1xd-2}
\end{equation}
We use \eqref{e:x1xd-2} to eliminate the numerator in the right-hand side of \eqref{e:phid_of_xd-1}. This yields
$$\varphi_d(\fx_{d-1}) = \frac{\Phi_{d-2}(\fx_1)}{(d-1)!} - \fx_{d-3}.$$

For the rest of this proof, assume $i = d$. Setting $j=d$ in \eqref{fxi_in_fTd}, applying $\varphi_d$ to the result, using induction, and dividing by $d$ yields
\begin{align}
    \varphi_d(\fx_d) &= \frac{\fx_1\big(\frac{\Phi_{d-2}(\fx_1)}{(d-1)!} - \fx_{d-3}\big) - 2(\fx_{d-2}-\fx_{d-4})}{d}.
    \label{phid_of_fxd}
\end{align}
Using Definitions~\ref{Fi_def} and~\ref{fTdDef}, we find that in $\fT_{d-2},$
\begin{equation}
    \fx_1\fx_{d-3} = (d-2)\fx_{d-2} + 2\fx_{d-4}.
    \label{fxd}
\end{equation}
In \eqref{phid_of_fxd}, we distribute terms in the numerator and eliminate $\fx_1\fx_{d-3}$ via \eqref{fxd}. This yields
\begin{align*}
    \varphi_d(\fx_d) &=  \frac{\frac{\fx_1\Phi_{d-2}(\fx_1)}{(d-1)!} - d\fx_{d-2}}{d} \\
    &= \frac{\fx_1\Phi_{d-2}(\fx_1)}{d!} - \fx_{d-2}.
\end{align*}

(ii) Similar to the proof of (i).
\end{proof}

Recall the ideal $\fS_d \subseteq \fT_d$ from Definition~\ref{fSd_def}. Our next general goal is to show that $\varphi_d(\fS_d) = \fS_{d-2}$. To do that, we will show that $\varphi_d(\fS_d) \subseteq \fS_{d-2}$ and $\fS_{d-2} \subseteq \varphi_d(\fS_d)$.

\begin{lem}\label{phi_of_esum1}
For $d \geq 2,$ the following $(i),(ii)$ hold.
\renumerate
    \item $\varphi_d\bigg(\sum_{i=0}^d \fe_i -1\bigg) = \sum_{i=0}^{d-2}\fe_i -1$.
    \item $\varphi_d\bigg(\sum_{i=0}^d \fe_i^* -1\bigg) = \sum_{i=0}^{d-2}\fe_i^* -1$.
\eenumerate
\end{lem}

\begin{proof}
(i) By Definition~\ref{phiDef},
\begin{equation*}
    \varphi_d\bigg(\sum_{i=0}^d \fe_i - 1\bigg) = \sum_{i=1}^{d-1}\fe_{i-1} -1 = \sum_{i=0}^{d-2}\fe_i -1.
    \label{element_(1)}
\end{equation*}

(ii) Similar to the proof of (i).
\end{proof}

\begin{lem}\label{hij_in_Pd-2}
Assume $d \geq 2$. For $0 \leq h,i,j \leq d$ such that $(h,i,j) \notin P_d$, the following (i), (ii) hold.
\renumerate
    \item $(h-1,i,j-1) \notin P_{d-2}$.
    \item $(h-1,i-2,j-1) \notin P_{d-2}$.
\eenumerate
\end{lem}

\begin{proof}
We consider the three cases in Corollary~\ref{permissible_conditions}. For convenience, we consider them in the order (ii), (iii), (i).

First, assume $h+i+j > 2d$. Then
\begin{align*}
    (h-1) + i + (j-1) &> 2(d-2), \\
    (h-1) + (i-2) + (j-1) &> 2(d-2).
\end{align*}
Thus (i) and (ii) hold.

Next, assume $h+i+j$ is odd. Then $(h-1)+i+(j-1)$ is odd and $(h-1)+(i-2)+(j-1)$ is odd. Thus (i) and (ii) hold.

For the rest of this proof, assume $h,i,j$ fail the triangle inequality. This leaves two subcases:
$$i > h+j, \qquad \qquad i < |h-j|.$$
First, assume $i > h+j$. Then
\begin{align*}
i &> (h-1) + (j-1), \\
i-2 &> (h-1) + (j-1).
\end{align*}
Hence $h-1,i,j-1$ and $h-1,i-2,j-1$ fail the triangle inequality. Thus (i) and (ii) hold.

Lastly, assume $i < |h-j|$. Then
\begin{align*}
i &< |(h-1)-(j-1)|, \\
i-2 &< |(h-1)-(j-1)|.
\end{align*}
Hence $h-1,i,j-1$ and $h-1,i-2,j-1$ fail the triangle inequality. Thus (i) and (ii) hold.

\end{proof}

\begin{lem}\label{phi_of_triple}
Assume $d \geq 2.$ For $0 \leq h,i,j \leq d$ such that $(h,i,j) \notin P_d$, the following $(i),(ii)$ hold.
\renumerate
    \item $\varphi_d(\fe_h^*\fx_i\fe_j^*) \in \fS_{d-2}$.
    \item $\varphi_d(\fe_h\fx_i^*\fe_j) \in \fS_{d-2}$.
\eenumerate
\end{lem}

\begin{proof}
(i) First note that if $h = 0$ or $h = d$ or $j = 0$ or $j = d$, then $\varphi_d(\fe_h^*\fx_i\fe_j^*) = 0$ by Definition~\ref{phiDef}. Thus for the remainder of this proof, we assume $1 \leq h,j \leq d-1$.

We consider the cases from Lemma~\ref{phi_of_xi} (i).

First, assume $i = 0$ or $i = 1$. By Definition~\ref{phiDef} and Lemma~\ref{phi_of_xi},
$$\varphi_d(\fe_h^*\fx_i\fe_j^*) = \fe_{h-1}^*\fx_i\fe_{j-1}^*.$$
By Lemma~\ref{hij_in_Pd-2}, $(h-1,i,j-1) \notin P_{d-2}$. Hence by Definition~\ref{fSd_def},
$$\fe_{h-1}^*\fx_i\fe_{j-1}^* \in \fS_{d-2}.$$

Next, assume $2 \leq i \leq d-2$. Then
$$\varphi_d(\fe_h^*\fx_i\fe_j^*) = \fe_{h-1}^*\fx_i\fe_{j-1}^* - \fe_{h-1}^*\fx_{i-2}\fe_{j-1}^*.$$
By Lemma~\ref{hij_in_Pd-2}, $(h-1,i,j-1) \notin P_{d-2}$ and $(h-1,i-2,j-1) \notin P_{d-2}$. Hence by Definition~\ref{fSd_def},
$$\fe_{h-1}^*\fx_i\fe_{j-1}^* - \fe_{h-1}^*\fx_{i-2}\fe_{j-1}^* \in \fS_{d-2}.$$

Next, assume $i = d-1$. Then
$$\varphi_d(\fe_h^*\fx_{i}\fe_j^*) = \frac{\fe_{h-1}^*\Phi_{d-2}(\fx_1)\fe_{j-1}^*}{(d-1)!} - \fe_{h-1}^*\fx_{i-2}\fe_{j-1}^*.$$
By Lemma~\ref{x1_minimal_poly_in_prime}, $\Phi_{d-2}(\fx_1) \in \fS_{d-2}$. By Definition~\ref{fSd_def} and Lemma~\ref{hij_in_Pd-2}, $\fe_{h-1}^*\fx_{i-2}\fe_{j-1}^*\in \fS_{d-2}$. Thus
$$\frac{\fe_{h-1}^*\Phi_{d-2}(\fx_1)\fe_{j-1}^*}{(d-1)!} - \fe_{h-1}^*\fx_{i-2}\fe_{j-1}^* \in \fS_{d-2}.$$

For the remainder of this proof, assume $i = d$. Then
$$\varphi_d(\fe_h^*\fx_i\fe_j*) = \frac{e_{h-1}^*\fx_1\Phi_{d-2}(\fx_1)\fe_{j-1}^*}{d!} - \fe_{h-1}^*\fx_{i-2}\fe_{j-1}^*.$$
By Lemma~\ref{x1_minimal_poly_in_prime}, $\Phi_{d-2}(\fx_1) \in \fS_{d-2}$. By Definition~\ref{fSd_def} and Lemma~\ref{hij_in_Pd-2}, $\fe_{h-1}^*\fx_{i-2}\fe_{j-1}^* \in \fS_{d-2}$. Thus
$$\frac{e_{h-1}^*\fx_1\Phi_{d-2}(\fx_1)\fe_{j-1}^*}{d!} - \fe_{h-1}^*\fx_{i-2}\fe_{j-1}^* \in \fS_{d-2}.$$

(ii) Similar to the proof of (i).
\end{proof}

We have now shown that $\varphi_d(\fS_d) \subseteq \fS_{d-2}$. Next we show that $\fS_{d-2} \subseteq \varphi_d(\fS_d)$. To that end, we include the following technical results.

\begin{lem}\label{phi_of_minpol}
Assume $d \geq 2$. The following (i), (ii) hold.
\renumerate
    \item $\Phi_{d-2}(\fx_1)(1-\fe_0-\fe_d) \in \fS_d$.
    \item $\Phi_{d-2}(\fx_1^*)(1-\fe_0^*-\fe_d^*) \in \fS_d$.
\eenumerate
\end{lem}

\begin{proof}
(i) Observe that
\begin{equation}\Phi_{d-2}(\fx_1)(1-\fe_0-\fe_d) = \Phi_{d-2}(\fx_1)\bigg(1 - \sum_{i=0}^d \fe_i\bigg) + \Phi_{d-2}(\fx_1)\sum_{i=1}^{d-1}\fe_i. \label{phi_difference_sums} \end{equation}
By Definition~\ref{fSd_def}, $1-\sum_{i=0}^d \fe_i \in \fS_d$. Hence
\begin{equation}
    \Phi_{d-2}(\fx_1)\bigg(1 - \sum_{i=0}^d \fe_i\bigg) \in \fS_d.
    \label{e:phid-2_big_sum_in_fSd}
\end{equation}
By Lemma~\ref{eigenvalue_formula} and Definition~\ref{Phi_def},
$$\Phi_{d-2}(\fx_1) = \sum_{i=1}^{d-1} (\fx_1-\theta_i).$$
By \eqref{e:x1_minus_theta} and Remark~\ref{r:psid_exists}, $(\fx_1 - \theta_i)\fe_i \in \fS_d$ ($1 \leq i \leq d-1$). Hence
\begin{equation}\Phi_{d-2}(\fx_1)\sum_{i=1}^{d-1}\fe_i \in \fS_d.\label{e:phid-2_in_fSd}\end{equation}
It follows from \eqref{phi_difference_sums}, \eqref{e:phid-2_big_sum_in_fSd}, and \eqref{e:phid-2_in_fSd} that
$$\Phi_{d-2}(\fx_1)(1-\fe_0-\fe_d) \in \fS_d.$$

(ii) Similar to the proof of (i).
\end{proof}

\begin{cor}
Assume $d \geq 2$. The following (i), (ii) hold.
\renumerate
    \item $\Phi_{d-2}(\fx_1) \in \varphi_d(\fS_d)$.
    \item $\Phi_{d-2}(\fx_1^*) \in \varphi_d(\fS_d)$.
\eenumerate
\label{Phid-2_in_fSd}
\end{cor}

\begin{proof}
(i) By Lemma~\ref{phi_of_minpol} (i), $\Phi_{d-2}(\fx_1)(1-\fe_0-\fe_d) \in \fS_d$. By Definition~\ref{phiDef} and Lemma~\ref{phi_of_xi},
\begin{equation*}
\varphi_d\big(\Phi_{d-2}(\fx_1)(1-\fe_0-\fe_d)\big) = \Phi_{d-2}(\fx_1).\end{equation*}

(ii) Similar to the proof of (i).
\end{proof}

\begin{lem}\label{h,i,j_in_Pd}
Assume $d \geq 2$. For $0 \leq h,i,j \leq d-2$ such that $(h,i,j) \notin P_{d-2}$, 
\begin{equation}(h+1,i-2r,j+1) \notin P_d \qquad \qquad (0 \leq r \leq \lfloor i/2 \rfloor), \label{i-2r}\end{equation}
or
\begin{equation}(h+1,i+2r,j+1) \notin P_d \qquad \qquad (1 \leq r \leq \lfloor (d-i)/2\rfloor). \label{i+2r} \end{equation}
\end{lem}

\begin{proof}
We consider the three cases in Corollary~\ref{permissible_conditions}. For convenience, we consider these cases in order (ii), (iii), (i).

First, assume that $h+i+j > 2(d-2)$. Then
$$(h+1)+(i+2r)+(j+1) > 2d \qquad \qquad (1 \leq r \leq \lfloor(d-i)/2\rfloor).$$
Thus \eqref{i+2r} holds.

Next, assume that $h+i+j$ is odd. Then $(h+1)+(i-2r)+(j+1)$ is odd ($0 \leq r \leq \lfloor i/2 \rfloor$). Thus \eqref{i-2r} holds.

For the rest of this proof, assume that $h,i,j$ fail the triangle inequality. This leaves two subcases:
$$ i > h+j, \qquad \qquad i < |h-j|.$$
First, assume $i > h+j$. Then
$$i + 2r > (j+1) + (h+1) \qquad \qquad (1 \leq r \leq \lfloor (d-i)/2 \rfloor).$$
Hence $h+1,i+2r,j+1$ fail the triangle inequality ($1 \leq r \leq \lfloor (d-i)/2 \rfloor$). Thus \eqref{i+2r} holds.

Lastly, assume $i < |h-j|$. Then
$$i -2r < |(h+1) - (j+1)| \qquad \qquad (0 \leq r \leq \lfloor i/2 \rfloor).$$
Hence $h+1,i-2r,j+1$ fail the triangle inequality ($0 \leq r \leq \lfloor i/2 \rfloor$). Thus \eqref{i-2r} holds.
\end{proof}

\begin{lem}\label{triple_in_image}
Assume $d \geq 2.$ For $0 \leq h,i,j \leq d-2$ such that $(h,i,j) \notin P_{d-2}$, the following $(i),(ii)$ hold.
\renumerate
    \item $\fe_h^*\fx_i\fe_j^* \in \varphi_d(\fS_{d})$.
    \item $\fe_h\fx_i^*\fe_j \in \varphi_d(\fS_{d})$.
\eenumerate
\end{lem}

\begin{proof}
(i) We consider the two cases in Lemma~\ref{h,i,j_in_Pd}.

First, assume \eqref{i-2r} holds. Then by Definition~\ref{fSd_def},
$$\fe_{h+1}^*\fx_{i-2r}\fe_{j+1}^* \in \fS_d, \qquad \qquad (0 \leq r \leq \lfloor i/2 \rfloor ).$$
By Definition~\ref{phiDef} and Lemma~\ref{phi_of_xi},
\begin{align}
    \varphi_d\bigg(\sum_{r=0}^{\lfloor i/2 \rfloor} \fe_{h+1}^*\fx_{i-2r}\fe_{j+1}^*\bigg) &= \sum_{r=0}^{\lfloor i/2 \rfloor-1}(\fe_h^*\fx_{i-2r}\fe_j^* - \fe_h^*\fx_{i-2r-2}\fe_j^*) + \fe_h^*\fx_{i-2\lfloor i/2 \rfloor}\fe_j^*.
    \label{varphid_of_i-2r_sum}
\end{align}
After expanding the sum and cancelling terms, the right-hand side of \eqref{varphid_of_i-2r_sum} becomes $\fe_h^*\fx_i\fe_j^*$. Thus $\fe_h^*\fx_i\fe_j^* \in \varphi_d(\fS_d).$

For the rest of this proof, assume \eqref{i+2r} holds. Then by Definition~\ref{fSd_def},
$$\fe_{h+1}^*\fx_{i+2r}\fe_{j+1}^* \in \fS_d \qquad \qquad (1 \leq r \leq \lfloor (d-i)/2 \rfloor ).$$
For notational convenience, define a polynomial $g \in \C[z]$ by
$$g = \begin{cases} \frac{1}{(d-1)!} &\text{if $d-i$ is odd;} \\ \frac{z}{d!} & \text{if $d-i$ is even.} \end{cases}$$
We have defined $g$ such that by Lemma~\ref{phi_of_xi},
$$\varphi_d(\fx_{i+2\lfloor(d-i)/2\rfloor}) = g(\fx_1)\Phi_{d-2}(\fx_1) - \fx_{i+2\lfloor (d-i)/2\rfloor - 2}.$$
Thus by Definition~\ref{phiDef} and Lemma~\ref{phi_of_xi},
\begin{equation}
\begin{split}
    &\varphi_d\bigg(\sum_{r=1}^{\lfloor(d-i)/2\rfloor} \fe_{h+1}^*\fx_{i+2r}\fe_{j+1}^*\bigg) \\
    =&\sum_{r=1}^{\lfloor(d-i)/2\rfloor-1}(\fe_h^*\fx_{i+2r}\fe_j^*-\fe_h^*\fx_{i+2r-2}\fe_j^*) + \fe_h^*g(\fx_1)\Phi_{d-2}(\fx_1)\fe_j^* - \fe_h^*\fx_{i+2\lfloor(d-i)/2\rfloor-2}\fe_j^*.
    \label{varphid_of_i+2r_sum}
\end{split}
\end{equation}
After expanding the sum and cancelling terms, the right-hand side of \eqref{varphid_of_i+2r_sum} becomes $-\fe_h^*\fx_i\fe_j^* + \fe_h^*g(\fx_1)\Phi_{d-2}(\fx_1)\fe_j^*.$ Hence
\begin{equation}-\fe_h^*\fx_i\fe_j^* + \fe_h^*g(\fx_1)\Phi_{d-2}(\fx_1)\fe_j^* \in \varphi_d(\fS_d).
\label{e:ehxiej+phi_stuff}
\end{equation}
By Corollary~\ref{Phid-2_in_fSd} (i) and the surjectivity of $\varphi_d,$
\begin{equation}
    \fe_{h}^*g(\fx_1)\Phi_{d-2}(\fx_1)\fe_{j}^* \in \varphi_d(\fS_d).
    \label{ehgphiej_in_phi_of_fSd}
\end{equation}
Therefore by \eqref{e:ehxiej+phi_stuff} and \eqref{ehgphiej_in_phi_of_fSd}, $\fe_h^*\fx_i\fe_j^* \in \varphi_d(\fS_d).$

(ii) Similar to the proof of (i).
\end{proof}

We have $\fS_{d-2} \subseteq \varphi_d(\fS_d)$ by Lemmas~\ref{phi_of_esum1} and~\ref{triple_in_image} together with the surjectivity of $\varphi_d$.

\begin{prop}\label{varphi_of_fSd}
Assume $d \geq 2.$ Then $\varphi_d(\fS_d) = \fS_{d-2}.$
\end{prop}

\begin{proof}\label{image_of_fSd}
We mentioned below Lemma~\ref{phi_of_triple} that $\varphi_d(\fS_d) \subseteq \fS_{d-2}$, and we mentioned below Lemma~\ref{triple_in_image} that $\fS_{d-2} \subseteq \varphi_d(\fS_d)$.
\end{proof}

We next consider how $\varphi_d$ induces an algebra homomorphism from $\cT_d \to \cT_{d-2}$.

\begin{prop}\label{hom_cTd_to_cTd-2}
Assume $d \geq 2.$ There exists an algebra homomorphism $\varphi_d':\cT_d \to \cT_{d-2}$ that sends
\begin{align*}
    e_0 &\mapsto 0, &e_0^* &\mapsto 0, \\
    e_i &\mapsto e_{i-1}, &e_i^*&\mapsto e_{i-1}^* \qquad \qquad (1 \leq i \leq d-1), \\
    e_d &\mapsto 0, &e_d^* &\mapsto 0.
\end{align*}
Moreover, $\varphi_d'$ is surjective, and $\ker(\varphi_d') = \psi_d(\fQ_d)$.
\end{prop}

\begin{proof}
We first consider the existence of $\varphi_d'$. By Lemma~\ref{fTd_hom_to_cTd}, Lemma~\ref{phid_exists}, and Proposition~\ref{varphi_of_fSd}, we have a surjective algebra homomorphism $\psi_{d-2} \circ \varphi_d:\fT_d \to \cT_{d-2}$ with kernel equal to $\fS_d + \fQ_d.$ This map induces an algebra isomorphism from the quotient algebra $\fT_d/(\fS_d + \fQ_d) \to \cT_{d-2}$; we say this isomorphism is canonical.

Let $q:\fT_d/\fS_d \to \fT_d/( \fS_d+\fQ_d)$ denote the quotient map, which we recall is an algebra homomorphism.

Recall the algebra isomorphism $p_d:\cT_d \to \fT_d/\fS_d$ from Definition~\ref{induced_iso_p}.

The following composition gives an algebra homomorphism from $\cT_d \to \cT_{d-2}$:
\begin{equation}
\begin{tikzcd}
\varphi_d': \quad \cT_d \arrow[r,"p_d"] & \fT_d/\fS_d \arrow[r,"q"] & \fT_d/(\fS_d + \fQ_d) \arrow[r,"can"] &  \cT_{d-2}.
\end{tikzcd}
\label{e:hom_diagram}
\end{equation}
We have shown that $\varphi_d'$ exists. With reference to \eqref{xi_less_0}, one routinely check that $\varphi_d'$ sends $e_i \mapsto e_{i-1}$ and $e_i^* \mapsto e_{i-1}^*$ ($0 \leq i \leq d$).

We next show that $\varphi_d'$ is surjective. This follows because each of the composition factors in \eqref{e:hom_diagram} is surjective.

Lastly, we consider the kernel of $\varphi_d'$. Inspection of \eqref{e:hom_diagram} shows that $\ker(\varphi_d') = p_d^{-1}(\fQ_d+\fS_d).$ By the construction of $p_d$, $p_d^{-1}(\fQ_d+\fS_d) = \psi_d(\fQ_d)$. Hence $\ker(\varphi_d') = \psi_d(\fQ_d)$.
\end{proof}


\begin{prop}\label{kerPhiIsu0}
Assume $d \geq 2$. The ideal $\cT_d u_0$ is equal to $\psi_d(\fQ_d).$ Moreover, $\cT_d u_0 = \ker(\varphi_d')$.
\end{prop}

\begin{proof}
We first consider the first assertion. Because $\psi_d$ is surjective, $\psi_d(\fQ_d)$ is equal to the two-sided ideal of $\cT_d$ generated by $e_0,e_d,e_0^*,e_d^*.$ Thus by Lemma~\ref{u0e0}, the ideal $\psi_d(\fQ_d) \subseteq \cT_du_0.$

Recall that $u_0 = 2^d\sum_{r=0}^d k_r^{-1}e_r^*e_0e_r^*.$ Thus
$$\cT_du_0 \subseteq \cT_de_0 \subseteq \psi_d(\fQ_d).$$
This proves the first assertion.

The second assertion follows by the first, together with Proposition~\ref{hom_cTd_to_cTd-2}.
\end{proof}

\begin{cor}\label{cTd-2_iso}
Assume $d \geq 2$. There exists an algebra isomorphism $\cT_d(1-u_0) \to \cT_{d-2}$ that sends
\begin{align*}
    e_0(1-u_0) &\mapsto 0, &e_0^*(1-u_0) &\mapsto 0, \\
    e_i(1-u_0) &\mapsto e_{i-1}, &e_i^*(1-u_0) &\mapsto e_{i-1}^* \qquad \qquad (1 \leq i \leq d-1), \\
    e_d(1-u_0) &\mapsto 0, &e_d^*(1-u_0) &\mapsto 0.
\end{align*}
\end{cor}

\begin{proof}
With reference to \eqref{xi_less_0}, Proposition~\ref{hom_cTd_to_cTd-2} implies that there exists an induced algebra isomorphism $\cT_d/\ker(\varphi_d') \to \cT_{d-2}$ which sends
$$e_i + \ker(\varphi_d') \mapsto e_{i-1}, \qquad \qquad e_i^*+\ker(\varphi_d') \mapsto e_{i-1}^* \qquad \qquad (0 \leq i \leq d).$$
By Proposition~\ref{kerPhiIsu0}, we know that $\ker(\varphi_d') = \cT_du_0$. Identifying the quotient algebra $\cT_d/\cT_du_0$ with $\cT_d(1-u_0)$ via the isomorphism in Corollary~\ref{Tdu0_quotient}, the result follows.
\end{proof}

\begin{cor}\label{cTd_sum_matrix_and_cTd-2}
Assume $d \geq 2$. There exists an algebra isomorphism
$$\cT_d\to M_{d+1}(\C)\oplus\cT_{d-2}.$$
\end{cor}

\begin{proof}
Follows from Corollaries~\ref{Td_iso_to_M_and_ideal} and~\ref{cTd-2_iso}.
\end{proof}

\begin{prop}\label{cTd_iso_Td}
Assume $d \geq 0.$ Then there exists an algebra isomorphism
$$\cT_d \to \bigoplus_{0 \leq r \leq \lfloor d/2\rfloor} M_{d+1-2r}(\C).$$
Moreover, the algebra $\cT_d$ is isomorphic to $T_d.$
\end{prop}

\begin{proof}
We first consider the first assertion. We proceed by induction on $d$. The base cases of $d=0$ and $d=1$ are addressed in Remark~\ref{T0_is_C} and Proposition~\ref{dEquals1}. Now assume $d \geq 2.$ By Corollary~\ref{cTd_sum_matrix_and_cTd-2} and induction, we have algebra isomorphisms
$$\cT_{d} \to M_{d+1}(\C) \oplus \cT_{d-2} \to  \bigoplus_{0 \leq r \leq d/2} M_{d+1-2r}(\C).$$
This completes the proof of the first assertion.

To prove the second assertion, compare this result to Proposition~\ref{T_sum_of_matrices}.
\end{proof}

We conclude with the main result of this paper.

\begin{thm}\label{main_result}
Assume $d \geq 0$. The map $\natural:\cT_d \to T_d$ from Lemma~\ref{hom_exists} is an algebra isomorphism.
\end{thm}

\begin{proof}
By Proposition~\ref{cTd_iso_Td}, $\cT_d$ and $T_d$ have the same dimension as algebras. Because $\natural$ is a surjective algebra homomorphism between two algebras of the same dimension, it is an algebra isomorphism.
\end{proof}


\section{Acknowledgments}
The author is presently a graduate student at the University of Wisconsin--Madison. He would like to thank his advisor, Paul Terwilliger, for suggesting this project, for his hours of mentoring, and for giving many valuable suggestions for this manuscript.


 \bigskip

 \bigskip

 \noindent Nathan Nicholson

 \noindent Department of Mathematics

 \noindent University of Wisconsin

 \noindent 480 Lincoln Drive

 \noindent Madison, WI 53706-1388 USA

 \noindent email: nlnicholson@wisc.edu

\end{document}